\documentclass[reqno, 12pt]{amsart}
\usepackage{cite}
\usepackage[english]{babel}
\usepackage{amssymb,amsmath,amsfonts,amsthm}
\usepackage[utf8]{inputenc}
\usepackage{cmap}
\usepackage[T1]{fontenc}
\usepackage[all]{xy}
\usepackage{mathbbol}
\usepackage{graphicx}
\usepackage{xcolor}
\usepackage{accents}
\usepackage{tikz-cd	}

\newtheorem{definition}{Definition}[section]
\newtheorem{theorem}{Theorem}[section]
\newtheorem{proposition}{Proposition}[section]
\newtheorem{lemma}{Lemma}[section]

\newtheorem{remark}{Remark}[section]

\DeclareMathOperator{\tr}{Tr}
\DeclareMathOperator{\rank}{rank}
\DeclareMathOperator{\SL}{SL}
\DeclareMathOperator{\GL}{GL}

\def\tr{\mathrm{Tr}}
\numberwithin{equation}{section}

\begin{document}

	\title[Calogero-Moser space and invariants]{Calogero-Moser spaces and the invariants of two  matrices of degree 3}
	\author{Z. Normatov}
	\address{[Zafar Normatov]  V.~I.~Romanovskiy Institute of Mathematics,  Uzbekistan Academy of Sciences
		University Street 4b, 100174 Tashkent, Uzbekistan}
	\email{z.normatov@mathinst.uz}
	\author{R. Turdibaev}
	\address{[Rustam Turdibaev]   V.~I.~Romanovskiy Institute of Mathematics,  Uzbekistan Academy of Sciences
		University Street 4b, 100174 Tashkent, Uzbekistan\\
		\newline AKFA University, 1st Deadlock 10, Kukcha Darvoza, 100095 Tashkent, Uzbekistan}
	\email{r.turdibaev@mathinst.uz}
	
	\begin{abstract}
		We find a minimal set of generators for the coordinate ring of Calogero-Moser space $\mathcal{C}_3$ and the algebraic relations among them explicitly. We give a new presentation for the algebra of  $3\times3$ invariant matrices involving the defining relations of $\mathbb{C}[\mathcal{C}_3]$. We find an explicit description of the commuting variety of  $3\times3$ matrices and its orbits under the action of the affine Cremona group.
	\end{abstract}
	
	\maketitle
	
	\section{Introduction}
	
	Let $\mathcal{M}_{n}$ be the $\mathbb{C}$-algebra of $n\times n$ matrices
	over $\mathbb{C}$, and let $d$ be a positive integer.
	The group $\mathrm{GL}_n(\mathbb{C})$ acts on the direct product  $\mathcal{M}_n^d$
	of $d$ copies of $\mathcal{M}_n$ by the simultaneous conjugation:
	\begin{equation}
		\label{actiongl}
		g \cdot (X_1,...,X_d) = (gX_1g^{-1},\dots,gX_dg^{-1})\, ,
		\quad   g \in \mathrm{GL}_n(\mathbb{C}) \, .
	\end{equation}

	This action induces an action of $\mathrm{GL}_n$ on the algebra
	$\mathbb{C}[\mathcal{M}_n^d]$ of polynomial functions on $\mathcal{M}_n^d$. By the general results of invariant theory of classical groups it is known that the algebra of invariant polynomials of the general linear group on $d$ matrices $\mathbb{C}[\mathcal{M}_n^d]^{\mathrm{GL}_n}$ is finitely generated. Moreover, a result of Procesi \cite{Pr}
	and Razmyslov \cite{Ra} states that $\mathbb{C}[\mathcal{M}_n^d]^{\mathrm{GL}_n}$
	is generated by the traces
	\[ \tr(Z_1\cdot \cdot \cdot Z_k)\, , \quad Z_1,...,Z_k \in \{X_1,...,X_d\}\, , \,
	1 \le k \le n^2 , \]
	and  every relation among them is a consequence of the Cayley-Hamilton theorem.

	It is well-known \cite{D} that traces $\tr(X), \tr(Y), \tr(X^2), \tr(XY), \tr(Y^2)$ generate $\mathbb{C}[\mathcal{M}_2\times \mathcal{M}_2]^{\mathrm{GL}_2}$ and there are no relations between the generators (see e.g.~\cite{DF}). A minimal system of eleven generators for $\mathbb{C}[\mathcal{M}_3\times \mathcal{M}_3]^{\GL_3}$ was found by Teranishi \cite{T} to be $\tr(X), \tr(Y)$ and
	\begin{align} 
		&	\tr(X^2), \tr(XY), \tr(Y^2),  \tr(X^3), 	\tr(X^2Y), \tr(XY^2), \tr(Y^3), \label{C32gen1}\\
		&	\tr(X^2Y^2), \tr(X^2Y^2XY) \label{C32gen2}.
	\end{align}
	A complicated algebraic relation of the square of $\tr(X^2Y^2XY)$ via the rest of the generators was given by Nakamoto \cite{Na}. Turns out, there is only one defining relation and it is much simpler if one uses another minimal set of  generators, which is the main result in \cite{ADS}: along with $\tr(X)$ and $\tr(Y)$, take traces of \eqref{C32gen1} with $X$ and $Y$ replaced with their traceless versions $A$ and $B$, correspondingly, and instead of (\ref{C32gen2}) take 
	\begin{align}
		v& = \tr(A^2B^2)-\tr(ABAB) \label{v}\\
		w &=  \tr(A^2B^2AB) -\tr(A^2BAB^2) \label{w}.
	\end{align}
	
	In this work we further simplify the defining relation found by authors of~\cite{ADS}. While we use the same set of minimal generators, we establish in section \ref{new_presentation} that the the defining relation is a polynomial of $$\tr(A^2),\tr(AB), \tr(B^2),v,w$$ and the the defining relations of the coordinate ring of the Calogero-Moser space~$\mathcal{C}_3$. 
	
	For an integer $n\ge 0$, let $\bar{\mathcal C}_n$ be the subset of
	$\mathcal{M}_n\times \mathcal{M}_n$ defined as
	\begin{equation*}
		\{ (X,Y) \in \mathcal{M}_n\times \mathcal{M}_n
		\mid \rank([X,Y] +I_n)=1\} ,
	\end{equation*}
	where $I_n$ is the identity $n\times n$ matrix.
	The action \eqref{actiongl} on $\mathcal{M}_n\times \mathcal{M}_n$ restricts to an action on
	$\bar{\mathcal C}_n$,
	and we define the $n$-th \textit {Calogero-Moser} space $\mathcal{C}_n$ to be the quotient variety $\bar{\mathcal{C}}_n / \! / \mathrm{GL}_n$.  Named after a class of integrable systems in classical mechanics, Calogero-Moser spaces play an important role in geometry and representation theory. These spaces were studied in detail by Wilson~\cite{W}, who proved that $\mathcal{C}_n$ is a smooth affine irreducible complex symplectic variety of dimension $2n$.

	The embedding of $\mathcal{C}_n$ into $(\mathcal{M}_{n}\times \mathcal{M}_{n}) /\negthickspace\, /\mathrm{GL}_n$
	allows us to view  $\mathbb{C}[\mathcal{C}_n]$, the coordinate ring of $\mathcal{C}_n$, as a quotient of
	the algebra $\mathbb{C}[\mathcal{M}_n\times \mathcal{M}_n]^{\mathrm{GL}_n}$.
	So it is natural to ask if a refinement of the Procesi-Razmyslov theorem gives an explicit presentation for $\mathbb{C}[\mathcal{C}_n]$. For presenting the coordinate ring of the Calogero-Moser space for $n=2$, it is more convenient to use traceless matrices. Denoting by $A$ and $B$ the traceless versions of $X$ and $Y$, correspondingly, one can show 
	$$ \mathbb{C}[\mathcal{C}_2] = \mathbb{C}[u_1,u_2,u_3,u_4,u_5]/(u_4^2-u_3u_5-1) \, , $$
	where $u_1=\tr(X), u_2=\tr(Y), u_3=\tr(A^2), u_4=\tr(AB), u_5=\tr(B^2)$. In this work we establish the presentation for $n=3$ and discuss its fruitful consequences in simplification of the defining relation of the ring of invariants of two matrices and the presentation of commuting variety of $3\times 3$ matrices.  
	
	The paper is organized as follows. First, in Section \ref{prelim} we establish that  
	\begin{equation}\label{generators}
		\begin{split}    
			& a_1=\tr(X),\ a_2=\tr(Y), \ a_3=\tr(A^2), \ a_4=\tr(AB), \ a_5=\tr(B^2),\\ 
			& a_6=\tr(A^3), \ a_7=\tr(A^2B), \ a_8=\tr(AB^2), \ a_9=\tr(B^3)
		\end{split}
	\end{equation}
	are generators of $\mathbb{C}[\mathcal{C}_3]$, where $A=X-\frac13a_1I_3, \ B=Y-\frac13a_2I_3$. In Section 
	\ref{relations_C3} we establish the following identities
	\begin{equation} \label{id1z}
		\begin{split}
			a_3a_9-2a_4a_8+a_5a_7=0\\
			a_5a_6-2a_4a_7+a_3a_8=0\\
			9a_3-a_3a_4^2+a_3^2a_5+6a_6a_8-6a_7^2=0\\
			9a_4-a_4^3+a_3a_4a_5+3a_6a_9-3a_7a_8=0\\
			9a_5-a_4^2a_5+a_3a_5^2+6a_7a_9-6a_8^2=0
		\end{split}
	\end{equation}
	for any pairs of matrices from $\mathcal{C}_3$. Later, in Section \ref{transitivity} we define a variety as a quotient of $\mathbb{C}^9$ by  relations (\ref{id1z}) and prove that the group of unimodular automorphisms of $\mathbb{C}[x,y]$, also known as the affine Cremona group, acts transitively on it. It is known by \cite[Theorem 1.3]{BW} that the same group acts transitively on $\mathcal{C}_3$. Hence, these varieties coincide and we obtain an explicit presentation of the coordinate ring of Calogero-Moser space $\mathcal{C}_3$ as
	$$ \mathbb{C}[a_1,a_2]\otimes \mathbb{C}[a_3,a_4,a_5,a_6,a_7,a_8,a_9]/I,$$
	where $I$ is generated by (\ref{id1z}).  Motivated by identities (\ref{id1z}), in section \ref{Calogero-Moser_type} for each point of a cuspidal curve $w^2+\frac4{27}v^3=0$ we define a Calogero-Moser type spaces $\mathcal{C}_{3,v}$. We establish that all of them are isomorphic to $\mathcal{C}_3$, unless $v=0$. In the later case, it is the commuting variety and Theorem \ref{main2} shows that its explicit presentation is very similar to the presentation of Calogero-Moser space $\mathcal{C}_3$. Moreover, we find that the commuting variety has three orbits under the affine Cremona group action in section \ref{C_30_last}.

	\section{Preliminaries}\label{prelim}
	\numberwithin{equation}{section} 
	Another presentation
	for $\mathcal{C}_n$ can be given as the space of all quadruples $(X,Y,c,r)$ where $X$ and $Y$ are $n\times n$ matrices and
	$c$ and $r$ are column and row vectors of the length $n$.
	Let $\tilde{\mathcal{C}}_n$ be the subspace
	consisting of all $(X,Y,c,r)$ satisfying
	\begin{equation}\label{vector-covector}
		XY-YX+I_n = c r .
	\end{equation}
	If we define the action of $\mathrm{GL}_n$  on $\tilde{\mathcal{C}}_n$  by
	$$ g \cdot (X,Y,c,r)=(gXg^{-1}, gYg^{-1}, gc, rg^{-1}) \, ,$$
	then there is a natural identification
	$ \mathcal{C}_n \cong \tilde{\mathcal{C}}_n/  \mathrm{GL}_n$ (see \cite{W} for details).
	
	Observe that equation (\ref{vector-covector}) also holds for traceless versions of $X$ and $Y$, the pair of matrices \[(A,B)=(X-\frac1n\tr(X)I_n,\ Y-\frac1n\tr(Y)I_n).\] Multiplying the both sides of the traceless version of (\ref{vector-covector}) by $A^k$ (by $B^k$) for any positive integer $k$, and considering traces one obtains
	\begin{align}
		\tr(A^k)=rA^kc\, \label{wAv}\\
		\tr(B^k)=rB^kc. \label{wBv}
	\end{align}
	
	Squaring both sides of the traceless version of (\ref{vector-covector}), we obtain 
	\[
	[A,B]^2+2[A,B]+I_n=crcr.\] Note that $rc=n$ and this leads to	
	\begin{equation}
		[A,B]^2=I_n+(n-2)cr. \label{vector-covector2}
	\end{equation}
	
	Consider trace in (\ref{vector-covector2}) to get \begin{equation}\label{ABAB-A^2B^2}\tr(ABAB)-\tr(A^2B^2)=\frac12\,n(n-1).
	\end{equation}
	
	Multiplying the traceless version of equality (\ref{vector-covector}) from the left by $AB$ (and $BA$) and considering traces, using (\ref{ABAB-A^2B^2}) one obtains
	\begin{align} \nonumber
		rABc=\tr(AB)+\frac12\, n(n-1)\\
		\label{wBAv}
		rBAc=\tr(AB)-\frac12\, n(n-1)
	\end{align}

	The following identities for traceless pairs of matrices $(A,B)\in \mathcal{C}_n$ are established in \cite[Proposition 2.2, 2.3]{N}:
	\begin{align}
		\tr(A^3B^2)&=\tr(A^2BAB) \label{A3B2=ABABA}\\
		\tr(A^2BAB^2) &= \tr(A^2B^2AB)-\frac{n(n-1)(n-2)}3  \label{A2BAB2-A2B2AB}\\
		\tr((AB)^3)& = \tr(A^2B^2AB)+(n-1)\tr(AB)-\frac{n(n-1)(n-2)}6 \label{(AB)^3}\\
		\tr(A^3B^3) &= \tr(A^2B^2AB)-(n-2)\tr(AB)-\frac{ n(n-1)(n-2)}6 \label{A3B3}
	\end{align}
	
	\begin{remark}\label{remark_9_gen} One can verify that 
		\[v=-\frac{1}{2}\tr([A,B]^2), \ w=\frac{1}{3}\tr([A,B]^3).\]
		Then equalities (\ref{ABAB-A^2B^2}) and (\ref{A2BAB2-A2B2AB}) show that the generators $v$ and $w$  introduced in (\ref{v})-(\ref{w}) on $\mathcal{C}_n$ are constants. For $n=3$, we get $v=-3$ and $w=2$. Thus, a minimal generating set of $\mathbb{C}[\mathcal{C}_3]$ admits at most nine elements. 
	\end{remark}
	
	To find the generators and defining relations of $\mathcal{C}_3$ we often use the Cayley-Hamilton theorem
	\begin{align} \label{tracelessCayleyHamilton}
		M^3& =\frac{1}{2}\tr(M^2)M+\frac13 \tr(M^3) I_3
	\end{align}
	for a traceless matrix $M\in \mathcal{M}_3$. Below we list some of the necessary traces in terms of notations (\ref{generators}).
	
	\begin{lemma}\label{20-lemma}
		For any traceless matrices $A,B\in \mathcal{M}_3$ we have
		\[
		\begin{split}
			&\tr(A^3B) =\frac12a_3a_4, \ \tr(AB^3) =\frac12a_5a_4, \\
			&\tr(A^3B^2)= \frac{1}{2}a_3a_8+\frac13 a_5a_6, \ \tr(A^2B^3)  = \frac{1}{2}a_5a_7+\frac13 a_3a_9,\\
			&\tr(A^3B^3) = \frac14a_3a_4a_5+\frac13 a_6a_9.
		\end{split}
		\]
		Furthermore, for traceless pairs of matrices $(A,B)\in \mathcal{C}_3$ we have
		\[
		\begin{split}
			&\tr(A^2B^2)=\frac16a_3a_5+\frac13a_4^2-1,  \\
			&\tr(A^4B^2)=\frac{1}{12}a_3^2a_5+\frac16a_3a_4^2-\frac12a_3 +\frac13a_6a_8.
		\end{split}
		\]
	\end{lemma}
	
	\begin{proof}
		The first part of the statement follows from (\ref{tracelessCayleyHamilton}) for $M=A$ (or $M=B$)  after directly considering traces, or multiplying both sides of (\ref{tracelessCayleyHamilton}) to corresponding matrices and then considering traces. Using Remark \ref{remark_9_gen} and \cite[Lemma 1.1]{ADS}, with similar technique as for the first part one can establish the other identities.
	\end{proof}


	Let us provide the following presentation given in \cite[Theorem 1.2]{ADS} of the algebra of invariants. We denote by $\mid A \mid$ the determinant of a matrix $A$.
	
	\begin{theorem}[\cite{ADS}] The algebra of invariants of two $3\times 3$ matrices is generated by (\ref{C32gen1}) and $v, w$ subject to the defining relation
		\begin{equation}\label{old_Drensky}
			w^2-\frac1{27}w_1+\frac29w_2-\frac4{15}w_3'-\frac1{90}w_3''-\frac13w_4+\frac23w_5+\frac13w_6+\frac4{27}w_7=0,
		\end{equation}
		where \[
		\begin{split}
			u&=\begin{vmatrix}
				a_3&a_4\\
				a_4&a_5
			\end{vmatrix}
			, \ w_1=u^3, \ w_2=u^2v, \ w_4=uv^2, \ w_7=v^3,\\
			w_5&=v \cdot \begin{vmatrix}
				a_3&a_4&a_5\\
				a_6&a_7&a_8\\
				a_7&a_8&a_9
			\end{vmatrix}, \ w_3'=u\ \cdot \begin{vmatrix}
				a_3&a_4&a_5\\
				a_6&a_7&a_8\\
				a_7&a_8&a_9
			\end{vmatrix}\\ 
			w_6&=\begin{vmatrix} a_6&a_8\\ a_7&a_9\end{vmatrix}^2-4\begin{vmatrix} a_6&a_8\\a_8&a_7 \end{vmatrix}\cdot \begin{vmatrix} a_6&a_7\\ a_7&a_8 \end{vmatrix} \\ 
			w_3''&=5(a_5^3a_6^2+a_3^3a_9^2)-30(a_5^2a_4a_7a_9 +a_3^2a_4a_8a_9)-2(a_4^3+3a_3a_4a_5)(9a_7a_8+a_6a_9)\\ 
			&+3(4a_5a_4^2+a_5^2a_3)(3a_7^2 +2a_8a_6)+3(4a_4^2a_3+a_3^2a_5)(3a_8^2+2a_7
			7a_9).\\
		\end{split}
		\]
	\end{theorem}

	In what follows we make use of the following group and its action on pairs of matrices. 
	
	\begin{definition}\label{Cremona}
		Let us denote by $G$ the group generated by unimodular automorphisms  of $\mathbb{C}[x,y]$:
		\begin{enumerate}
			\item[(i)] $\Phi_p(x)=x+p(y)$ and $\Phi_p(y)=y$, where $p\in  \mathbb{C}[x]$,
			\item[(ii)] $\Psi_q(x)=x$ and $\Psi_q(y)=y+q(x)$, where $q \in \mathbb{C}[x]$.
		\end{enumerate}
	\end{definition} 
	The action of this groups on $(\mathcal{M}_n\times \mathcal{M}_n) /\!/ \GL_n$ is defined similarly as follows \cite{BW}:
	\begin{enumerate}
		\item[(i)] $\Phi_p:(X,Y)\mapsto (X+p(Y),Y)$, where $p\in  \mathbb{C}[x]$,
		\item[(ii)] $\Psi_q:(X,Y)\mapsto (X,Y+q(X))$, where $q \in \mathbb{C}[x]$.
	\end{enumerate}
	Note that the action of $G$ commutes with the $\GL_n$-action on pairs of matrices. There is a natural decomposition of $\GL_3$-sets:
	\begin{equation}    \label{decomposition}
		(\mathcal{M}_3\times \mathcal{M}_3) /\!/\GL_3=
		\mathcal{C}_3 \cup \mathcal{C}'_3 \cup \mathcal{C}''_3,
	\end{equation}
	where  
	\[
	\begin{split}
		\mathcal{C}'_3=\{(X,Y)\in M_3\times M_3 \mid \rank([X,Y]+I_3)=2\}/\negthickspace /{\GL_3}\\
		\mathcal{C}''_3=\{(X,Y)\in M_3\times M_3 \mid \rank([X,Y]+I_3)=3\}/\negthickspace /{\GL_3}
	\end{split} \ .
	\] 
	Note that $\rank([X,Y]+I)$ is never zero. Since $[X+p(Y),Y]=[X,Y]$, the action  of $G$ preserves commutators and the decomposition (\ref{decomposition}) is a $G$-invariant decomposition.  Now for a given pair $(X,Y)\in (\mathcal{M}_{3}\times \mathcal{M}_{3})  / \negthickspace  / \GL_3$ the action of $G$ induces a group action on the variety of tuples $(a_1,\dots,a_9,v,w)$ that satisfy (\ref{old_Drensky}). By Remark \ref{remark_9_gen}, it follows that the action of $G$ keeps $v$ and $w$ invariant. Hence, while considering the action of $G$ on the coordinate ring, we can restrict ourselves on the first $9$ entries. 
	
	One of the main ingredients of the proof of our main result is the following fact \cite[Theorem 1.3]{BW}  by Berest and Wilson.
	\begin{theorem}[\cite{BW}] The action of $G$ on each of the spaces $\mathcal{C}_n$ is transitive. 
	\end{theorem}
	
	
	\section{Relations in $\mathcal{C}_3$} \label{relations_C3}

	\begin{proposition}
		For traceless pairs of matrices $(A,B)\in \mathcal{C}_3$ we have
		\begin{align} \label{A^4B^2}
			\tr(A^4B^2)& =\tr(A^3BAB)-\frac{3}{2}\tr(A^2) \\ \label{(A^2B)^2}
			\tr(A^3BAB)& =\tr((A^2B)^2)+\frac12\tr(A^2)
		\end{align}
	\end{proposition}
	
	\begin{proof}
		Multiplying the traceless version of  (\ref{vector-covector}) by $A^3$ from the left and by $B$ from the right, one obtains
		\begin{align}\label{Trace A^4B^2}
			\tr(A^4B^2)&=\tr(A^3BAB)-\tr(A^3B)+rBA^3c.
		\end{align}
		Multiplying equation (\ref{tracelessCayleyHamilton}) for $M=A$ from the right by $B$  we have
		\begin{align} \label{BA^3}
			BA^3&=\frac12 \tr(A^2)BA+\frac13\tr(A^3)B.
		\end{align}
		Then multiplying (\ref{BA^3}) from the left by $r$ and from the right by $c$,
		one has
		\begin{equation}\label{wBA^3v}
			rBA^3c =\frac{1}{2}\tr(A^2)rBAc+\frac13 \tr(A^3)rBc.
		\end{equation}
		Similarly, we get
		\begin{equation*}
			rA^3Bc =\frac{1}{2}\tr(A^2)rABc+\frac13 \tr(A^3)rBc.
		\end{equation*}
		Substituting (\ref{wBv}), (\ref{wBAv}),  (\ref{wBA^3v}) into (\ref{Trace A^4B^2}) and using Lemma \ref{20-lemma} yields (\ref{A^4B^2}). 
		Similarly,  multiplying (\ref{vector-covector}) by $A^2$ from the left and by $AB$ from the right and considering traces of both sides, we have
		\begin{align}\label{A3BAB}
			\tr(A^3BAB)=\tr((A^2B)^2)-\tr(A^3B)+rABA^2c.
		\end{align}
		Multiplying (\ref{vector-covector}) by $A^2$ from the right, we obtain
		\begin{equation}\label{ABA^2}
			ABA^2-BA^3+A^2=crA^2.
		\end{equation}
		Then multiplying (\ref{ABA^2}) from the left to $r$ and from the right to $c$, using (\ref{wAv}), (\ref{wBA^3v}), (\ref{wBv}), (\ref{wBAv}) consecutively deduces
		\begin{equation}\label{wABA^2v}
			\begin{split}
				rABA^2c&= rBA^3c+2rA^2c\\
				&=\frac{1}{2}\tr(A^2)rBAc+\frac13 \tr(A^3)rBc+2\tr(A^2)\\
				&=\frac{1}{2}\tr(A^2)(\tr(AB)-3)+2\tr(A^2)\\
				&=\frac{1}{2}\tr(A^2)\tr(AB)+\frac{1}{2}\tr(A^2).
			\end{split}
		\end{equation}

		Substituting  (\ref{wABA^2v}) into (\ref{A3BAB}) and using Lemma \ref{20-lemma} we obtain (\ref{(A^2B)^2}).
	\end{proof}

	\begin{proposition}\label{main_5eq}
		For traceless pairs of matrices $(A,B)\in \mathcal{C}_3$ the equations in (\ref{id1z}) 
		in the notations (\ref{generators}) hold.
	\end{proposition}
	\begin{proof}
		Consider straight-forward expansions $\tr((A+Bt)^5)=\displaystyle \sum_{i=0}^5 c_i t^i$ and $\tr((A+Bt)^6)=\displaystyle\sum_{i=0}^6 d_i t^i$. We list some of the coefficients:
		\[
		\begin{split} 
			c_2 & = 5(\tr(A^3B^2)+\tr(A^2BAB))\\
			c_3& = 5(\tr(A^2B^3)+\tr(ABAB^2))\\
			d_2&= 6\tr(A^4B^2)+6\tr(A^3BAB)+3\tr((A^2B)^2)\\
			d_3& = 6\tr(A^3B^3)+6\tr(A^2B^2AB)+6\tr(B^2A^2BA)+2\tr((AB)^3)\\
			d_4& = 6\tr(A^2B^4)+6\tr(ABAB^3)+3\tr((AB^2)^2)
		\end{split}
		\]
		
		On the other hand, the equation (\ref{tracelessCayleyHamilton}) for $M=A+tB$ yields
		\[
		\begin{split}
			\tr((A+tB)^5)=& \frac56 \tr((A+tB)^2) \tr((A+tB)^3)\\
			\tr((A+tB)^6)=&\frac14 \tr^3((A+tB)^2)+\frac13\tr((A+tB)^3)\\
		\end{split}
		\]
		and expanding by powers of $t$, we obtain the following different expressions for the same coefficients:
		\begin{equation*}
			\begin{split} 
				c_2 & = \frac56 (a_5a_6+6a_4a_7+3a_3a_8)\\
				c_3& = \frac56(a_3a_9+6a_4a_8+3a_5a_7)\\
				d_2&=  3 a_3a_4^2 +\frac{3}{4} a_3^2a_5 +3 a_7^2+2 a_6 a_8\\
				d_3& = 2 a_4^3+3 a_3  a_4a_5+6 a_7 a_8+\frac23 a_6 a_9\\
				d_4& = \frac{3}{4} a_3 a_5^2+3  a_4^2a_5+3 a_8^2+2 a_7 a_9
			\end{split}
		\end{equation*}
		Comparing expression for $c_2$, Lemma \ref{20-lemma} and  (\ref{A3B2=ABABA}) implies  $a_5a_6-2a_4a_7+a_3a_8=0$. Similarly, by considering expressions for $c_3$ one establishes $a_3a_9-2a_4a_8+a_5a_7=0$.
		
		Now comparing expressions for $d_2$ and applying Lemma \ref{20-lemma}  in (\ref{A^4B^2}) and  (\ref{(A^2B)^2}), we deduce 
		$$9a_3-a_3a_4^2+a_3^2a_5+6a_6a_8-6a_7^2=0.$$
		Similarly, by considering $d_4$ one establishes $9a_5-a_4^2a_5+a_3a_5^2+6a_7a_9-6a_8^2=0$. 
		
		Finally, substituting (\ref{A2BAB2-A2B2AB}), (\ref{(AB)^3}), (\ref{A3B3}) into expressions that are equal to $d_3$ and using Lemma \ref{20-lemma} implies  $3a_6a_9+9a_4+a_3a_4a_5-3a_7a_8-a_4^3=0$.
	\end{proof}

	\section{A new presentation of $\mathbb{C}[\mathcal{M}_3\times \mathcal{M}_3]^{\GL_3}$}\label{new_presentation}
	
	In this section we present the defining relation  (\ref{old_Drensky}) of $\mathbb{C}[\mathcal{M}_3\times \mathcal{M}_3]^{\GL_3}$ with  the same set of minimal generators as known by \cite{ADS} in a different and more simple form.  Let us define the following polynomials:
	\begin{equation}\label{r's}
		\begin{split}
			r_1&=\begin{vmatrix}a_3&a_4\\a_8&a_9\end{vmatrix}-\begin{vmatrix}a_4&a_5\\a_7&a_8\end{vmatrix},  \quad r_2=\begin{vmatrix} a_3&a_4\\a_7&a_8\end{vmatrix}-\begin{vmatrix}a_4&a_5\\a_6&a_7\end{vmatrix}\\
			r_3&=a_3\Big(\begin{vmatrix} a_3&a_4\\ a_4&a_5\end{vmatrix}-3v\Big)+6\cdot \begin{vmatrix} a_6&a_7\\ a_7&a_8 \end{vmatrix} \\  
			r_4&=a_4\Big(\begin{vmatrix} a_3&a_4\\ a_4&a_5\end{vmatrix}-3v\Big)+3\cdot \begin{vmatrix}
				a_6&a_7\\ a_8&a_9 \end{vmatrix} \\
			r_5&=a_5\Big(\begin{vmatrix} a_3&a_4\\ a_4&a_5\end{vmatrix}-3v\Big)+6\cdot \begin{vmatrix}
				a_7&a_8\\ a_8&a_9 \end{vmatrix}
		\end{split}.
	\end{equation}
	These polynomials are motivated by relations (\ref{id1z}) that hold for $\mathbb{C}[\mathcal{C}_3]$. Indeed, if $v=-3$, the polynomials $r_1,\dots, r_5$ are equal to the polynomials in the left-hand side of equalities (\ref{id1z}). 
	
	\begin{theorem}
		The defining relation (\ref{old_Drensky}) coincides with 
		\begin{equation}\label{new_Drensky_rel}
			w^2+\frac{4v^3}{27}-\frac1{27}(r_3r_5-r_4^2)-\frac1{18} (a_3r_1^2-2a_4r_1r_2+a_5r_2^2)=0,
		\end{equation}
		where $r_1,r_2,r_3,r_4,r_5$ are given in (\ref{r's}).    
	\end{theorem}
	\begin{proof}
		Straight-forward verification with the help of any computer algebra software establishes equality of the left-hand sides of (\ref{old_Drensky}) and (\ref{new_Drensky_rel}) for arbitrary pairs of matrices $(X,Y)\in \mathcal{M}_3\times \mathcal{M}_3$.
	\end{proof}

	\section{Calogero-Moser space $\mathcal{C}_3$ and similar varieties} \label{Calogero-Moser_type}
	
	Since $\mathcal{C}_3$ is embedded in $(\mathcal{M}_3\times \mathcal{M}_3)/\negthickspace /{\GL_3} $,  relation (\ref{new_Drensky_rel}) holds in $\mathbb{C}[\mathcal{C}_3]$. Recall, for Calogero-Moser space by Remark \ref{remark_9_gen} we have $v=-3$ and $w=2$ and furthermore, $r_1=r_2=r_3=r_4=r_5=0$ by Proposition \ref{main_5eq}. This results in relation (\ref{new_Drensky_rel}) becoming a trivial relation in  $\mathbb{C}[\mathcal{C}_3]$. This is not coincidental since turns out there are no other defining relations for $\mathbb{C}[\mathcal{C}_3]$ and the presentation is as follows. 
	\begin{theorem}\label{main}
		The traces in (\ref{generators}) form a minimal generating set of $\mathbb{C}[\mathcal{C}_3]$ with the defining relations (\ref{id1z}). Explicitly, we have 
		$$\mathbb{C}[\mathcal{C}_3]\cong \mathbb{C}[a_1,a_2]\otimes \mathbb{C}[a_3,a_4,a_5,a_6,a_7,a_8,a_9]/I,$$
		where $I$ is generated by (\ref{id1z}).
	\end{theorem}
	
	The proof of the Theorem \ref{main} is given in section \ref{transitivity}. Motivated by the explicit presentation above, let us consider similar spaces.

	\begin{definition}
		Denote by $\mathcal{C}_{3,v}$ the subspace of $\mathbb{C}[\mathcal{M}_3\times \mathcal{M}_3]^{\GL_3}$ whose coordinate ring is 
		$\mathbb{C}[a_1,a_2]\otimes \dfrac{\mathbb{C}[a_3,a_4,a_5,a_6,a_7,a_8,a_9]}{(r_1,r_2,r_3,r_4,r_5)}$.
	\end{definition}
	From (\ref{new_Drensky_rel}) note that $\mathcal{C}_{3,v}$ is defined for each point of the cuspidal curve \[w^2+\frac{4v^3}{27}=0.\] 
	We are interested in cases $v\neq -3$, since $\mathcal{C}_{3,-3}$ is just $\mathcal{C}_3$ and in the next statement we find for which values of $v$ the varieties $\mathcal{C}_{3,v}$ are contained in $\mathcal{C}'$ and $\mathcal{C}''$ of the decomposition (\ref{decomposition}).

	\begin{proposition}  $\mathcal{C}_{3,v} \subset \mathcal{C}'_3$ if and only if $v+w=-1$ with $v\neq -3$.
	\end{proposition}
	\begin{proof} We can characterize $\mathcal{C}'_3$ as
		\[ \{(X,Y) \mid 
		\det([X,Y]+I_3)=0, \ v\neq -3\}.
		\]
		Apply well-known formula 
		$$ \det(M)=\frac16 \tr^3(M)-\frac12\tr(M^2)\tr(M)+\frac13\tr(M^3)
		$$
		for $M=[X,Y]+I_3$. After simplifications we obtain that $\mathcal{C}'_3$ is defined by
		$$0=1-\frac12\tr([X,Y]^2)+\frac13\tr([X,Y]^3)
		$$
		and using $[X,Y]=[A,B]$, it simplifies to $1+v+w=0$. 
	\end{proof}

	We establish that these varieties are isomorphic to either the Calogero-Moser space $\mathcal{C}_3$ or to the variety of commuting $3\times 3$ matrices.

	\begin{theorem}
		For all $v\neq 0$ the varieties $\mathcal{C}_{3,v}$ and $\mathcal{C}_3$ are isomorphic.
	\end{theorem}
	\begin{proof}
		Consider a map $\phi\colon (\mathcal{M}_3\times \mathcal{M}_3) /\!/\GL_3 \to (\mathcal{M}_3\times \mathcal{M}_3) /\!/\GL_3$ defined by $(X,Y)\mapsto (\alpha X,Y)$. Then 
		\[
		\begin{split}
			(a_1,a_2, a_3,a_4,a_5,a_6,a_7,a_8,a_9)&\mapsto (\alpha a_1,a_2,\alpha^2 a_3,\alpha a_4, a_5, \alpha^3 a_6,\alpha^2 a_7,\alpha a_8,a_9)\\
			(v,w)& \mapsto (\alpha^2v,\alpha^3w)\\
			(r_1,r_2,r_3,r_4,r_5) & \mapsto (\alpha^2r_1,\alpha^3 r_2, \alpha^4 r_3, \alpha^3 r_4, \alpha^2 r_5).
		\end{split}
		\]
		and relation (\ref{new_Drensky_rel}) is mapped to itself with a multiple $\alpha^6$. Assuming $\alpha^2=\frac{-3}{v}$ we obtain $v\mapsto -3$ and $\mathcal{C}_{3,v}$ is isomorphic to $\mathcal{C}_3$.
	\end{proof}
	
	In case $v=0$ we have the following 
	\begin{theorem}
		The coordinate ring of the commuting variety of $3\times 3$ matrices is $\mathbb{C}[a_1,a_2]\otimes \mathbb{C}[a_3,a_4,a_5,a_6,a_7,a_8,a_9]/J,$ where $J$ is generated by
		\begin{equation} \label{id2z}
			\begin{split}
				a_3a_9-2a_4a_8+a_5a_7\\
				a_5a_6-2a_4a_7+a_3a_8\\
				a_3^2a_5-a_3a_4^2+6a_6a_8-6a_7^2\\
				a_3a_4a_5-a_4^3+3a_6a_9-3a_7a_8\\
				a_3a_5^2-a_4^2a_5+6a_7a_9-6a_8^2
			\end{split}
		\end{equation}
	\end{theorem}
	\begin{proof} If matrices $X$ and $Y$ commute, by Remark \ref{remark_9_gen} we get $v=w=~\negthickspace0$. Similar to the proof of  Proposition \ref{main_5eq} we can establish the relations (\ref{id2z}). 
		
		Conversely, consider the variety $\mathcal{C}_{3,0}$. By Theorem \ref{main2} there are exactly three $G$-orbits. The following pairs of matrices are the corresponding representatives:
		\[(O_3,O_3), \ (A, B), \ (A,C), 
		\]
		where 
		$$A=\begin{bmatrix}
			0&1&0\\
			0&0&1\\
			0&0&0
		\end{bmatrix}, B=\begin{bmatrix}
			-\frac{3i+\sqrt{3}}{2\sqrt[6]{3^5}}&0&0\\
			0&\frac{3i-\sqrt{3}}{2\sqrt[6]{3^5}}&0\\
			0&0&\frac{1}{\sqrt[3]{3}}
		\end{bmatrix}, C=\begin{bmatrix}
			-\frac{3i+\sqrt{3}}{2\sqrt[6]{3^5}}&0&0\\
			0&\frac{3i-\sqrt{3}}{2\sqrt[6]{3^5}}&0\\
			0&0&\frac{1}{\sqrt[3]{3}}
		\end{bmatrix}.$$ 
		
		Conjugating  the matrices $A$,  $[A,B]$ and $[A,C]$ with a diagonal matrix $\textrm{diag}(t^3,t^2, t)$, we get $\begin{bmatrix}
			0 & t & 0\\
			0 & 0 & t\\
			0 & 0 & 0\\
		\end{bmatrix}$, while $B$ and $C$ are unchanged. Thus, both $(A,B)$ and $(A,C)$ fall into the closure of the orbit corresponding to a commuting pairs of matrices.  Since $G$- and $\GL_n$-actions commute, we obtain that $\mathcal{C}_{3,0}$ as a GIT-quotient is equal to the commuting variety. 
	\end{proof}
	
	\begin{remark} Note, in \cite{Do} and \cite{Va} an isomorphism  $\mathcal{C}_{3,0}\cong (\mathbb{C}[x,y]^{\otimes 3} )^{S_3}$ is established. 
	\end{remark}
	
	In the last section we provide the proof of the following result on the orbits of the commuting variety under the affine Cremona group action.
	\begin{theorem}\label{main2}
		The $G$-orbits of $\mathcal{C}_{3,0}$ are
		$$[(0,0,0,0,0,0,0,0,0)], \ [(0,0,0,0,\sqrt[3]{6},0,0,0,1)], \  [(0,0,0,0,0,0,0,0,1)].$$
	\end{theorem}

	\section{Proof of Theorem \ref{main} }\label{transitivity}
	
	For any point $A=(a_i)_{1\leq i \leq 9} \in \mathbb{C}^9$ we have $$(\Psi_{-\frac{a_1}3} \circ \Phi_{-\frac{a_2}3})(A)=(0,0,a_3,\dots,a_9).$$
	Hence, when studying $G$-orbits, we can neglect the first two coordinates and work with seven-tuples instead.    
	
	Note that $\Phi_p$ fixes $a_1,a_3,a_6$ and $\Psi_p$ fixes $a_2,a_5,a_9$. We frequently use in what follows the explicit action of $\Phi_p$ for a quadratic monomial $p(x)=\alpha x^2$:
	\[
	\begin{split}
		a_4&\mapsto a_4+\alpha(a_6+\frac{2}{3}a_1a_3)\\
		a_5&\mapsto a_5+\alpha(2a_7+\frac{4}{3}a_1a_4)+\alpha^2(\frac{4}{9}a_1^2a_3+\frac{1}{6}a_3^2+\frac{4}{3}a_1a_6)\\
		a_7&\mapsto a_7+\frac{1}{6}\alpha(a_3^2+4a_1a_6)\\
		a_8&\mapsto a_8+\alpha(\frac{1}{3}a_3a_4+\frac{4}{3}a_1a_7)+\alpha^2(\frac{2}{9}a_1a_3^2+\frac{4}{9}a_1^2a_6+\frac{1}{6}a_3a_6)\\
		a_9&\mapsto a_9+\alpha(a_4^2-\frac{1}{2}a_3a_5+2a_1a_8-3)+\alpha^2(\frac{2}{3}a_1a_3a_4+a_4a_6+\frac{4}{3}a_1^2a_7-\frac{1}{2}a_3a_7)\\
		& +\alpha^3(\frac{2}{9}a_1^2a_3^2-\frac{1}{36}a_3^3+\frac{8}{27}a_1^3a_6+\frac{1}{3}a_1a_3a_6+\frac{1}{3}a_6^2)\\
	\end{split}
	\]
	For $\Psi_{\alpha x^2}$ the results are symmetric.

	For a matrix $M=\begin{bmatrix}
		\alpha & \beta\\
		\gamma & \delta
	\end{bmatrix} \in \SL_2$ (that is $\alpha \delta-\beta \gamma=1$) consider a map $\Theta_M\colon  \mathcal{C}_3\to \mathcal{C}_3$ defined by 
	$(X,Y)\mapsto (\alpha X+\beta Y,\gamma X+\delta Y)$. This map belongs to $G$ since it is a composition of the automorphisms of type (i) and (ii) of Defintion \ref{Cremona} for some linear polynomials $p$ and $q$. The seven-tuple  $(a_3,\dots,a_9)$ under this map changes as follows:
	\[\begin{split}
		\Theta_M( a_3)= &  \alpha^2 a_3+2\alpha\beta a_4+\beta^2 a_5\\
		\Theta_M( a_4)= &  \alpha\gamma a_3+(\alpha\delta+\beta\gamma) a_4+\beta\delta a_5\\
		\Theta_M( a_5)= &  \gamma^2 a_3+2\gamma\delta a_4+\delta^2 a_5\\
		\Theta_M( a_6)= &  \alpha^3 a_6+3\alpha^2\beta a_7+3\alpha\beta^2 a_8+\beta^3 a_9\\
		\Theta_M( a_7)= &  \alpha^2\gamma a_6+2\alpha\beta\gamma a_7+\alpha^2\delta a_7+\beta^2\gamma a_8+2\alpha\beta\delta a_8+\beta^2\delta a_9\\
		\Theta_M(a_8 )= &  \alpha\gamma^2 a_6+2\alpha\delta\gamma a_7+\gamma^2\beta a_7+\delta^2\alpha a_8+2\gamma\beta\delta a_8+\beta\delta^2 a_9\\
		\Theta_M(a_9 )= &   \gamma^3 a_6+3\gamma^2\delta a_7+3\gamma\delta^2 a_8+\delta^3 a_9
	\end{split} .
	\]
	Let us define $\mathcal{D}_3$ to be a subset of $\mathbb{C}^7$, points of which are defined by relations~(\ref{id1z}). By Proposition \ref{main_5eq} we have $\mathcal{C}_3\subseteq \mathbb{C}^2\times\mathcal{D}_3$.  We will prove that they coincide. Since $G$ acts transitively on $\mathcal{C}_3$, it suffices to show that $G$ acts transitively on $\mathcal{D}_3$. In order to establish that, we need the following lemmas. 
	
	\begin{lemma}\label{a_7=0}
		Let $A=(0,a_4,a_5,a_6,0,a_8,a_9)\in \mathcal{D}_3$. Then there exists $g\in G$ such that either $g(A)$ is a zero point or  $g(A)=(0,a_4',a_5',a_6',a_7',a_8',a_9')$  with $a_7'\neq 0$. 
	\end{lemma}
	
	\begin{proof} Equalities (\ref{id1z}) yield
		$$a_5a_6=0, \ a_4a_8=0, \ 3a_6a_9+9a_4-a_4^3=0, \ 9a_5-a_4^2a_5-6a_8^2=0, \ a_6a_8=0 .$$	
		Consider several cases.

		\textit{Case 1. } Let $a_8\neq0$. Then we have 
		$$a_4=a_6=0, \ 3a_5-2a_8^2=0$$
		and $A=(0,0,\frac{2}{3}a_8^2,0,0,a_8,a_9)$. One can compute, that $$A':=\Phi_{\frac{1}{3}a_9t^2}(A)=(0,0,\frac{2}{3}a_8^2,0,0,a_8,0).$$ 
		Then consider $\Psi_{\alpha t^2}(A')=(b_3,0,b_5,b_6,0,b_8,0)$, where $b_3,b_5,b_6,b_8$ are non-zero polynomials in $\alpha$. Hence, we can choose a value of $\alpha$ such that $b_3b_5b_8\neq0$. 
		
		Consider $\Theta_M((b_3,0,b_5,b_6,0,b_8,0))$ with $M=\begin{bmatrix}
			-\sqrt{\frac{b_5}{b_3}} & 1\\
			\frac14 & -\frac54 \sqrt{\frac{-b_3}{b_5}}
		\end{bmatrix}$. Then the fifth coordinate is equal to $b_8$, which is not zero. 
		
		\
		
		\textit{Case 2. } Let $a_8=0$. 
		
		\textit{Case 2.1.} Let $a_6\neq0$. Then we have $a_5=0$, \ $3a_6a_9+9a_4-a_4^3=0$ and $A=(0,a_4,0,a_6,0,0,a_9)$. This implies
		$$\Phi_{-\frac{a_4}{a_6}t^2}(A)=(0,0,0,a_6,0,0,0).$$
		Now  $\Psi_{\frac{1}{3}a_6t^2}$ sends the point above to a zero point. 
		
		\textit{Case 2.2.} Let $a_6=0$. Then we have $9a_4-a_4^3=0, \ 9a_5-a_4^2a_5=0$ and $A$ is either  $(0,0,0,0,0,0,a_9)$ or $(0,\pm3,a_5,0,0,0,a_9)$. Applying $\Phi_{\frac{1}{3}a_9t^2}(A)$ we may assume $a_9=0$. 
		
		\textit{Case 2.2.1.} Let $a_5=0$.  Note that $\Psi_{t^2}(A)= (0\pm3,0,6,0,0,0)$ and further applying $\Theta_M$ with $M=\begin{bmatrix}
			1 & 0\\
			1 & 1
		\end{bmatrix}$ we obtain the point $(0\pm3,\pm6,6,6,0,0)$. 
		
		\textit{Case 2.2.2.} Let $a_5\neq0$. Then  $\Phi_{t^2}(A)=(0\pm3,a_5,0,0,0,6)$ and applying $\Theta_M$ for $M=\begin{bmatrix}
			\mp\frac{a_5}{6} & 1\\
			-1\mp\frac{a_5}{6} & 1
		\end{bmatrix}$ we obtain the point $(0\mp3,\mp6,6,6,6,6)$.
	\end{proof}

	\begin{lemma}\label{000abcdef}
		Let $A=(0,a_4,a_5,a_6,a_7,a_8,a_9)\in \mathcal{D}_3$. Then there is $g\in G$ such that $g(A)$ is a zero point.
	\end{lemma}
	
	\begin{proof} Due to Lemma \ref{a_7=0} we may assume that $a_7\neq 0$. Moreover, if $a_5=0$ then $a_4=0$ due to relations \eqref{id1z}. If $a_5\neq 0$ then for $p(t)=-\frac{a_5}{2a_7}t^2$, taking into account equalities (\ref{id1z}) we find that  $$\Phi_{p}(A)=(0,0,0,a_6,a_7,a_8,a'_9)=:A'.$$
		
		Next, we have $\Theta_M(A')=(0,0,0,0,0,0,a''_9)$ for 
		$M=\begin{bmatrix}
			-\frac{a_8}{a_7} & 1\\
			-1-\frac{a_8}{a_7} & 1
		\end{bmatrix}$. Finally, $\Phi_{\frac{1}{3}a''_9t^2}(\Theta_M(A'))$ is a zero point. 
	\end{proof} 
	
	\begin{theorem}
		The group action on the set $\mathcal{D}_3$ is a transitive group action.
	\end{theorem}

	\begin{proof} Let us show that the orbit of a zero point coincides with $\mathcal{D}_3$. Consider an arbitrary point $B=(b_3,b_4,b_5,b_6,b_7,b_8,b_9)\in \mathcal{D}_3$. If $b_3=0$, then by Lemma \ref{000abcdef} this point is in the orbit of the zero point.  Hence, suppose $b_3\neq 0$. 
		
		\textit{Case 1.} Let  $b_5=0$.
		
		For a matrix
		$M=\begin{bmatrix}
			0 & 1\\
			-1 & 1
		\end{bmatrix}$ we obtain $\Theta_M(B')=(0,b_4,b_3,b_6,b_7,b_8,b_9)$ and we apply Lemma \ref{000abcdef}. 
		
		\textit{Case 2.} Let $b_5\neq0$. Then  $b_4+\sqrt{b_4^2-b_3b_5} \neq 0$,  otherwise $b_3=0$.  For a matrix
		$M=\begin{bmatrix}
			\alpha & 1\\
			0 & 1/\alpha
		\end{bmatrix} $ with $\alpha=\frac{-b_5}{b_4+\sqrt{b_4^2-b_3b_5}}$, we have $\Theta_M(B')=(0,a_4,a_5,a_6,a_7,a_8,a_9)$. By Lemma \ref{000abcdef} we are done. 
	\end{proof}
	
	As a result, the statement of Theorem \ref{main} follows. 
	
	\section{Proof of Theorem \ref{main2}}\label{C_30_last} 
	
	Starting from a point $(0,0,\dots,0)\in \mathcal{C}_{3,0}$, using the action of $G$ we can only reach  points in the form $(a_1,a_2,0,\dots,0)$.  Similar to the arguments of section \ref{transitivity}, in order to find the $G$-orbits of $\mathcal{C}_{3,0}$ we can omit the first two coordinates and work with the variety $\mathcal{D}_{3,0}$ of seven-tuples.

	\begin{lemma}\label{a_3=0}
		Let $A=(a_3,a_4,a_5,a_6,a_7,a_8,a_9)\in \mathcal{D}_{3,0}$ with $a_3\neq0$. Then there exists $g\in G$ such that  $g(A)=(0,b_4,b_5,b_6,b_7,b_8,b_9)$.
	\end{lemma}
	\begin{proof} The proof is a straight-forward verification of the following: if $a_5=0$ use $\Theta_M$ with 
		$M=\begin{bmatrix}
			0 & 1\\
			-1 & 1
		\end{bmatrix}$, and if  $a_5\neq0$ use $\Theta_M$ with 
		$M=\begin{bmatrix}
			\alpha & 1\\
			0 & \alpha^{-1}
		\end{bmatrix} $, where  $\alpha=\frac{-a_5}{a_4+\sqrt{a_4^2-a_3a_5}}$.
	\end{proof}

	\begin{lemma}\label{b_7=0}
		Let  $A\in \mathcal{D}_{3,0}\setminus \{[(0,0,\sqrt[3]{6},0,0,0,1)]\cup [(0,0,0,0,0,0,0)]\}$ be an arbitrary point of the form $(0,a_4,a_5,a_6,0,a_8,a_9)$. Then there exists $g\in G$ such that  $g(A)=(0,a_4',a_5',a_6',a_7',a_8',a_9')$  with $a_7'\neq 0$.
	\end{lemma}

	\begin{proof} Relations (\ref{id2z}) yield
		\[a_5a_6=0, \ a_4a_8=0, \ 3a_6a_9-a_4^3=0, \ -a_4^2a_5-6a_8^2=0, \ a_6a_8=0.
		\]
		If $a_8\neq0$, we obtain $a_4=a_6=0$ which implies $a_8=0$. Hence, we assume  $a_8=0$ and consider several cases.
		
		\textit{Case 1.} Let $a_6\neq0$. Then we have $a_5=0$, \ $3a_6a_9-a_4^3=0$ from \eqref{id2z} and  $A=(0,a_4,0,a_6,0,0,a_9)$. This implies
		$$\Phi_{-\frac{a_4}{a_6}t^2}(A)=(0,0,0,a_6,0,0,0).$$
		Now applying $\Theta_M$ with $M=\begin{bmatrix}
			1& -1\\
			1 & 0
		\end{bmatrix}$ we obtain the point $(0,0,0,a_6,a_6,a_6,a_6)$.
		
		\textit{Case 2.} Let $a_6=0$. Then we have $a_4=0$ and $A$ is  $(0,0,a_5,0,0,0,a_9)$. If $a_5=0$ and $a_9=0$ then the point is zero point. Contradiction.
		
		\textit{Case 2.1.} Let $a_5=0, a_9\neq0$. Then $A$ is  $(0,0,0,0,0,0,a_9)$. Now applying $\Theta_M$ with $M=\begin{bmatrix}
			1& 1\\
			0 & 1
		\end{bmatrix}$ we obtain the point $(0,0,0,a_9,a_9,a_9,a_9)$.
		
		\textit{Case 2.2.} Let $a_5\neq0, a_9=0$. Then $A$ is  $(0,0,a_5,0,0,0,0)$. Now applying $\Theta_M$ with $M=\begin{bmatrix}
			1& 1\\
			0 & 1
		\end{bmatrix}$ we obtain the point $(a_5,a_5,a_5,0,0,0,0)$.
		\[
		\Psi_{\alpha t^2}(a_5,a_5,a_5,0,0,0,0)=(0,a_5,a_5,\frac{2\sqrt{-6a_5^3}}{3},\frac{\sqrt{-6a_5^3}}{3},\frac{\sqrt{-6a_5^3}}{6},0)
		\]
		where $\alpha=\sqrt{-\frac6{a_5}}$. 
		
		\textit{Case 2.3.} Let $a_5\neq0, a_9\neq0$. Then $A$ is  $(0,0,a_5,0,0,0,a_9)$. Now applying $\Theta_M$ with $M=\begin{bmatrix}
			1& 1\\
			0 & 1
		\end{bmatrix}$ we obtain the point $(a_5,a_5,a_5,a_9,a_9,a_9,a_9)$. We can get
		\[
		\begin{split}
			\Phi_{\alpha^2t}(a_5,a_5,a_5&,a_9,a_9,a_9,a_9)=(a_5,a_5+a_9\alpha,a_5+2a-9\alpha+\frac{a_5^2}6\alpha^2,a_9,a_9+\frac{a_5^2}6\alpha,\\
			&a_9+\frac{a_5^2}3\alpha+\frac16a_5a_9\alpha^2,
			a_9+\frac{a_5^2}2\alpha+\frac12a_5a_9\alpha^2+(-\frac{a_5^3}{36}+\frac{a_9^2}3)\alpha^3)
		\end{split}
		\]
		
		If $a_5^3-6a_9^2\neq0$, then we can send it to the desired point by $\Psi_{\beta t}$, where $\beta=\frac{-6 a_5-6 a_9 \alpha -\sqrt{6(6 a_9^2 \alpha ^2-a_5^3 \alpha ^2)}}{6 a_5+12 a_9 \alpha +a_5^2 \alpha ^2}$ since we can choose $\alpha$ satisfying desired conditions.
		
		If $a_5^3-6a_9^2=0$, then applying $\Theta_M$ with $M=\begin{bmatrix}
			\sqrt[3]{a_9}& 1\\
			0 & \frac1{\sqrt[3]{a_9}}
		\end{bmatrix}$ we obtain the point $(0,0,\sqrt[3]{6},0,0,0,1)$, which is a contradiction. 
	\end{proof}

	\begin{proof}[Proof of Theorem \ref{main2}] 	Let $A\in \mathcal{D}_{3,0}\setminus \{[(0,0,\sqrt[3]{6},0,0,0,1)]\cup [(0,0,0,0,0,0,0)]\}$. Then we will prove that $A$ is in the orbit of the point $(0,0,0,0,0,0,1)$.
		
		By Lemma \ref{a_3=0} and Lemma \ref{b_7=0} we may assume that the point $A$ is of the form $(0,b_4,b_5,b_6,b_7,b_8,b_9)$ with $b_7\neq 0$. Now we consider two cases.
		
		\textit{Case 1}. Let $b_5\neq0$. For $p(t)=-\frac{b_5}{2b_7}t^2$, taking into account relations (\ref{id2z}) we find that  $$A':=\Phi_{p}(A)=(0,0,0,b_6,b_7,b_8,\frac{b_8^2}{b_7}).$$
		Then we have the following relation $b_6b_8-b_7^2=0$. Thus we may assume $b_6b_8\neq0$, otherwise we get a contradiction to $b_7\neq0$.
		For a matrix
		$M=\begin{bmatrix}
			\frac{b_8\sqrt[3]{b_6}}{b_7} & -\sqrt[3]{b_6}\\
			(1+\frac{b_8}{b_7})\frac{1}{\sqrt[3]{b_6}} & -\frac{1}{\sqrt[3]{b_6}}
		\end{bmatrix}$  we compute that  $\Theta_M(A')=(0,0,0,0,0,0,1)$.
		
		\textit{Case 2.} Let $b_5=0$. In this case it must be $b_4=0$ by equalities \eqref{id2z}. Hence the point is of the form $(0,0,0,b_6,b_7,b_8,b_9)$. Now we can repeat the above process (used to the point $A'$) to get the result.
	\end{proof}

\end{document}